\documentclass[12pt]{article}
\usepackage{epsf}
\usepackage{amsbsy,amsmath}
\usepackage{amsfonts}
\usepackage{amssymb}
\usepackage{eucal}
\usepackage{graphics,mathrsfs}
\usepackage{amsthm}
\usepackage{secdot}
\newtheorem{thm}{\bf Theorem}[section]
\newtheorem{cor}{\bf Corollary}[section]
\newtheorem{lem}{\bf Lemma}[section]

\numberwithin{equation}{section}

\newsavebox{\savepar}

\begin{document}
\title{\bf Composition operators on Orlicz-Sobolev spaces}
\author{Ratan Kumar Giri, Debajyoti Choudhuri \footnote{Corresponding author} \\
 {\it{\small Dept. of Mathematics, National Institute of Technology, Rourkela, Odisha, India }}\\
{\it{\small e-mails: giri90ratan@gmail.com,
choudhurid@nitrkl.ac.in}}
}
\date{\small\today}

 \maketitle
\begin{abstract}
The kernel of composition operator $C_T$ on Orlicz-Sobolev space is
obtained. Using the kernel, a necessary and  a sufficient condition
for injectivity of composition operator $C_T$ has been established.
Composition operators on Orlicz-Sobolev space with finite ascent as
well as infinite ascent have been characterized.
\end{abstract}

\section{Introduction and Preliminaries}
\par Let $\Omega$ be an open subset of the Euclidean space
$\mathbb{R}^n$ and $(\Omega,\Sigma,\mu)$ be a $\sigma$-finite
complete measure space, where $\Sigma$ is a $\sigma$-algebra of
subsets of $\Omega$ and $\mu$ is a non-negative measure on $\Sigma$.
Let $\varphi: \mathbb{R} \rightarrow \mathbb{R}$ be an N-function
\cite{Ames80, Ames94}, i.e., an even, convex and continuous function
satisfying $\varphi(x)=0$ if and only if $x=0$ with
 $\displaystyle {\lim_{x\rightarrow 0}\frac{\varphi(x)}{x}=0}$ and
$\displaystyle {\lim_{x\rightarrow
\infty}\frac{\varphi(x)}{x}=\infty}.$ Such a function $\varphi$ is
known as Orlicz function. Let $L^{0}(\Omega)$ denote the linear
space of all equivalence classes of $\Sigma$-measurable functions on
$\Omega$, where we identify any two functions are equal if they
agree $\mu$-almost everywhere on $\Omega$. The Orlicz space
$L^\varphi (\Omega)$ is defined as the set of all functions $f\in
L^{0}(\Omega)$ such that $\displaystyle{\int_{\Omega}\varphi(\alpha
|f|)d\mu} < \infty$ for some $\alpha>0$. The space $L^\varphi
(\Omega)$ is a Banach space with respect to the Luxemburg norm
defined by
$$ ||f||_\varphi=\inf \left \{k>0:\int_\Omega\varphi\left(\frac{|f|}{k}\right)\leq 1\right\}.$$
If $\varphi(x)= x^p$, $1\leq p <\infty$, then $L^{\varphi}(\Omega)=
L^p(\Omega)$, the well known Banach space of $p$-integrable function
on $\Omega$ with $||f||_\varphi=
\left(\frac{1}{p}\right)^{\frac{1}{p}} ||f||_p$. An Orlicz function
$\varphi$ is said to satisfy the $\Delta_2$-condition, if there
exists constants $k>0$, $u_\circ\geq0$ such that $\varphi(2u)\leq k
\varphi(u)$ for all $u\geq u_\circ$. If $ \tilde{L}^\varphi(\Omega)$
denotes the set of all function $f\in L^{0}(\Omega)$ such that
$\displaystyle{\int_{\Omega}\varphi(|f|)d\mu} < \infty$, then one
has $L^\varphi (\Omega)= \tilde{L}^\varphi(\Omega)$, when the Orlicz
function $\varphi$ satisfies the $\Delta_2$ condition. Here, the set
$ \tilde{L}^\varphi(\Omega)$ is known as Orlicz class. We define the
closure of all bounded measurable functions in $L^\varphi (\Omega)$
by $E^\varphi(\Omega)$. Then $E^\varphi(\Omega)\subset L^\varphi
(\Omega)$ and $E^\varphi(\Omega)= L^\varphi (\Omega)$ if and only if
$\varphi$ satisfies $\Delta_2$ condition. For further literature
concerning Orlicz spaces, we refer to Kufener, John and Fucik
\cite{Ames79}, Krasnoselskii $\&$ Rutickii\cite{Ames80} and Rao
\cite{Ames94}.
\par The Orlicz-Sobolev space $W^{1,\varphi}(\Omega)$ is defined as
the set of all functions $f$ in Orlicz space $L^\varphi (\Omega)$
whose weak partial derivative $\frac{\partial f}{\partial x_i}$ ( in
the distribution sense ) belong to $L^\varphi (\Omega)$, for all
$i=1,2,\cdots,n$. It is a Banach space with respect to the norm:
$$||f||_{1,\varphi}= ||f||_\varphi + \displaystyle{\sum_{i=1}
^n\big|\big|\frac{\partial f}{\partial x_i}\big|\big|_\varphi}$$ If
we take an Orlicz function $\varphi(x)=|x|^p$, $1\leq p<\infty$,
then $W^{1,\varphi}(\Omega)= W^{1,p}(\Omega)$ with
$||f||_{1,\varphi}= \left(\frac{1}{p}\right)^{1/p}||f||_{1,p}$,
i.e., the corresponding Orlicz-Sobolev space $
W^{1,\varphi}(\Omega)$ becomes classical Sobolev space of order one.
In fact, these spaces are more general than the usual Lebesgue or
Sobolev spaces. For more details on Sobolev and Orlicz-Sobolev
spaces, we refer to Adam \cite{sob1}, Rao \cite{Ames94} and Arora,
Datt and Verma \cite{sob}.
\par Let $T:\Omega\rightarrow\Omega$ be a measurable transformation, that
is, $T^{-1}(A)\in\Sigma$ for any $A\in\Sigma$. If $\mu(T^{-1}(A))=0$
for any $A\in\Sigma$ with $\mu(A)=0$, then $T$ is called as
nonsingular measurable transformation. This condition implies that
the measure $\mu\circ T^{-1}$, defined by $\mu\circ T^{-1}(A)
=\mu(T^{-1}(A))$ for $A\in\Sigma$, is absolutely continuous with
respect to $\mu$ i.e., $\mu\circ T^{-1}\ll\mu$. Then the
Radon-Nikodym theorem implies that there exist a non-negative
locally integrable function $f_T$ on $\Omega$ such that
$$\mu\circ T^{-1}(A)=\int_{A}f_T(x)d\mu(x)\,\,\,\quad\mbox{for}
\,\, A\in\Sigma.$$ It is known that any nonsingular measurable
transformation $T$ induces a linear operator (Composition operator)
$C_T$ from $L^{0}(\Omega)$ into itself which is defined as
$$C_Tf(x)= f(T(x)),\quad x\in\Omega,\quad f\in L^{0}(\Omega).$$
Here, the non-singularity of $T$ guarantees that the operator $C_T$
is well defined. Now, if the linear operator $C_T$ maps from
Orlicz-Sobolev space $W^{1,\varphi}(\Omega)$ into itself and is
bounded, then we call $C_T$ is a composition operator in
$W^{1,\varphi}(\Omega)$ induced by $T$. A major application of
Orlicz-Sobolev spaces can be found in PDEs \cite{orlsv}.
\par The composition operators received considerable attention over
the past several decades especially on some measurable function
spaces such as $L^p$-spaces, Bergman spaces and Orlicz spaces, such
that these operators played an important role in the study of
operators on Hilbert spaces. The basic properties of composition
operators on measurable function spaces have been studied by many
mathematicians. For a flavor of composition operators on different
spaces we refer to \cite{Ames102}, \cite{Ames81}, \cite{blum74} and
\cite{Ames101} and the references therein. The boundedness and
compactness of the composition operator $C_T$ on Orlicz-Sobolev
space $W^{1,\varphi}(\Omega)$ has been characterized in the paper
due to Arora, Datt and Verma \cite{sob}. Regarding the boundedness
of the composition operator $C_T$ on Orlicz-Sobolev space
$W^{1,\varphi}(\Omega)$ into itself, we have the following two
important results \cite{sob}.
\begin{lem}
Let $f_T,\,\, \frac{\partial T_k}{\partial x_i}\in L^\infty(\Omega)$
with $||\frac{\partial T_k}{\partial x_i}||_\infty \leq M$, for some
$M>0$ and for all $i,k=1,2,\cdots,n$, where $T= (T_1,
T_1,\cdots,T_n)$ and $\frac{\partial T_k}{\partial x_i}$ denotes the
partial derivative (in the classical sense). Then for each $f$ in
$W^{1,\varphi}(\Omega)$, we have $C_T(f)\in W^{1,\varphi}(\Omega)$
and if the Orlicz function $\varphi$ satisfies $\Delta_2$ condition,
then the first order distributional derivatives of $f\circ T$, given
by
$$\displaystyle{\frac{\partial}{\partial x_i}(f\circ T)=\sum_{k=1}^n\left(\frac{\partial f}{\partial x_k}\circ
T\right)\frac{\partial T_k}{\partial x_i}}\,\,,$$ for
$1,2,\cdots,n$, are in $L^\varphi(\Omega)$.
\end{lem}
\begin{thm}
Suppose that conditions are hold as given in the previous lemma.
Then the composition operator $C_T$ on Orlicz-sobolev space
$W^{1,\varphi}(\Omega)$ is bounded and the norm of $C_T(f)$
satisfies the following inequality:
$$||C_T(f)||_{1,\varphi}\leq ||f_T||_\infty (1+n
M)||f||_{1,\varphi}.$$
\end{thm}
In the present paper, we are going to present some classical
properties of Composition operators on Orlicz-Sobolev space
$W^{1,\varphi}(\Omega)$, which have not been proved earlier. We
divide this paper into two sections - section $2,3$. In the section
$2$, we establish a necessary and sufficient condition for the
injectivity of the composition operator on Orlicz-Sobolev space. And
in the last section, composition operators on Orlicz-Sobolev space
$W^{1,\varphi}(\Omega)$ with finite ascent as well as infinite
ascent are studied.

\section{Composition operator on Orlicz-Sobolev spaces}
We begin by the defining Orlicz-Sobolev space. For a given Orlicz
function $\varphi$, the corresponding Orlicz-Sobolev space is given
by
$$W^{1,\varphi}(\Omega) =\left\{ f\in
L^\varphi(\Omega): \frac{\partial f}{\partial x_i}\in
L^\varphi(\Omega)\,\,\, \mbox{for}\,\, i=1,2,\cdots,n\right\}.$$
Assume that the conditions as given in the lemma $1.1.$ holds. Let
$C_T:W^{1,\varphi}(\Omega)\rightarrow W^{1,\varphi}(\Omega)$ be a
nonzero composition operator and $\Omega_\circ= \{x\in \Omega:
f_T(x)=\frac{d\mu\circ T^{-1}}{d\mu}(x)=0\}$. Now consider the
subset
$$W^{1,\varphi}(\Omega_\circ)= \left\{f\in W^{1,\varphi}(\Omega):
f(x)=0\,\,\mbox{for}\,\,\Omega\setminus \Omega_\circ\right\}.$$ Then
the set $\Omega_\circ$ is obviously measurable. Note that if
$\mu(\Omega\setminus\Omega_\circ)=0$, then $\frac{d\mu\circ
T^{-1}}{d\mu}(x)=0$ almost everywhere and $||\frac{d\mu\circ
T^{-1}}{d\mu}||_\infty=0$. Thus in this case, the corresponding
composition operator $C_T$ will be the zero operator (by the theorem
$1.1.$). Hence for a nonzero composition operator $C_T$, we have
$\mu(\Omega\setminus\Omega_\circ)>0$. We start with the following
result.
\begin{lem}
If $ f \in W^{1,\varphi}(\Omega_\circ)$ then the weak derivative of
$f$, $\frac{\partial f}{\partial x_i}(x)=0$ on
$\Omega\setminus\Omega_\circ$, for every $i=1,2,\cdots,n.$
\end{lem}
\begin{proof}
Let $ f \in W^{1,\varphi}(\Omega_\circ)$. Then $f\in
L^\varphi(\Omega)$ and $\frac{\partial f}{\partial x_i} \in
L^\varphi(\Omega)$ for $1=1,2,\cdots,n$. Since weak derivative of
$f$ exists on $\Omega$, hence $f$ has weak derivative on
$\Omega\setminus\Omega_\circ\subset\Omega$. Now $f(x)=0$ on
$\Omega\setminus\Omega_\circ$ implies that
$\displaystyle{\int_{\Omega\setminus\Omega_\circ} f(x)\frac{\partial
\phi(x)}{\partial x_i} d\mu=0}$ for all test function $\phi(x) \in
C^\infty _0(\Omega)$ and $i=1,2,\cdots,n.$ We have,
$$\int_{\Omega\setminus\Omega_\circ} f(x)\frac{\partial
\phi(x)}{\partial x_i} d\mu = - \int_{\Omega\setminus\Omega_\circ}
\frac{\partial f(x)}{\partial x_i} \phi(x) d\mu\,\,,\forall\,\phi
\in C^\infty _0 (\Omega\setminus\Omega_\circ).$$ Since $C^\infty _0
(\Omega\setminus\Omega_\circ)\subset C^\infty_0(\Omega)$ hence
\begin{eqnarray*}
\int_{\Omega\setminus\Omega_\circ} \frac{\partial f(x)}{\partial
x_i} \phi(x) d\mu & = & - \int_{\Omega\setminus\Omega_\circ} f(x)
\frac{\partial \phi(x)}{\partial x_i} d\mu\\
 & = &0
\end{eqnarray*}
for all $\phi \in C^\infty _0 (\Omega\setminus\Omega_\circ)$. Now as
$\mu(\Omega\setminus\Omega_\circ)>0$, it follows that weak
derivative of $f$ is zero, i.e.,  $\frac{\partial f}{\partial
x_i}(x)=0$ on $\Omega\setminus\Omega_\circ$.
\end{proof}
\begin{thm}
Let $C_T$ be a composition operator on Orlicz-Sobolev space
$W^{1,\varphi}(\Omega)$. Then $ker\,
C_T=W^{1,\varphi}(\Omega_\circ)$.
\end{thm}
\begin{proof}
Let $f\in ker\, C_T$. Then $f\circ T=0$ in $W^{1,\varphi}(\Omega)$.
This implies that $||f\circ T||_{1,\varphi}=0$. But $||f\circ
T||_{1,\varphi}= ||f\circ T||_\varphi + \displaystyle{\sum_{i=1}
^n\big|\big|\frac{\partial}{\partial x_i}(f\circ
T)\big|\big|_\varphi}$. Therefore, we have $||f\circ T||_\varphi=0$
and $||\frac{\partial}{\partial x_i}(f\circ T)||_\varphi=0$ for
$i=1,2,\cdots,n$. This shows that $f\circ T =0$ in $L^\varphi$.
Hence, there exists $\alpha>0$ such that
\begin{eqnarray}
0= \int_\Omega \varphi(\alpha|f\circ T|)d\mu = \int_\Omega
\varphi(\alpha |f|)\frac{d\mu\circ T^{-1}}{d\mu} d\mu
\end{eqnarray}
Suppose $S_f= \{x\in \Omega: f(x)\neq 0\}$. Then, from above it
follows that $\frac{d\mu\circ T^{-1}}{d\mu}|_{S_f}=0$. Since
\begin{eqnarray*}
 W^{1,\varphi}(\Omega_\circ) & = & \left\{ f\in
W^{1,\varphi}(\Omega):
f(x)=0\,\,\mbox{for}\,\,\Omega\setminus\Omega_\circ\right\}\\
&= & \left\{ f\in W^{1,\varphi}(\Omega): S_f\subset
\Omega_\circ\right\}\\
& = & \left\{f\in W^{1,\varphi}(\Omega):\frac{d\mu\circ
T^{-1}}{d\mu}\big|_{S_f}=0\right\}\,\,,
\end{eqnarray*}
hence, $f\in W^{1,\varphi}(\Omega_\circ)$. Therefore, $ker\, C_T
\subseteq W^{1,\varphi}(\Omega_\circ)$.\\
Conversely, suppose that $f\in W^{1,\varphi}(\Omega_\circ)$. Then $
f\in L^\varphi(\Omega)$. Hence there exists $\alpha>0$ such that
$\displaystyle{\int_\Omega \varphi(\alpha |f|)d\mu <\infty}$, for
some $\alpha>0$. Now we have,
\begin{eqnarray*}
\int_\Omega \varphi(\alpha|f\circ T|)d\mu & = & \int_\Omega
\varphi(\alpha |f|) \frac{d\mu\circ T^{-1}}{d\mu}d \mu \\
& = & \int_{\Omega\setminus\Omega_\circ} \varphi(\alpha |f|)
\frac{d\mu\circ T^{-1}}{d\mu}d \mu  + \int_{\Omega_\circ}
\varphi(\alpha |f|) \frac{d\mu\circ T^{-1}}{d\mu}d \mu\\
& = & 0\,\,\,\,\,\,\,\,[\,\because\,\,f(x)=0\,\, \mbox{on}\,\,
\Omega\setminus\Omega_\circ\,\,\mbox{and}\,\, \frac{d\mu\circ
T^{-1}}{d\mu}(x)=0\,\,\mbox{on}\,\, \Omega_\circ\,]
\end{eqnarray*}
This implies that $f\circ T=0$ in $L^\varphi(\Omega)$ and hence
$||f\circ T||_\varphi=0$. Now by the lemma $1.1.$ we have,
$$\displaystyle{\frac{\partial}{\partial x_i}(f\circ T)=\sum_{k=1}^n\left(\frac{\partial f}{\partial x_k}\circ
T\right)\frac{\partial T_k}{\partial x_i}}.$$ Since, $f\in
W^{1,\varphi}(\Omega_\circ)$ and $C_T$ is a composition operator on
$W^{1,\varphi}(\Omega)$, hence weak derivative of $f\circ T$,
$\frac{\partial }{\partial x_i}(f\circ T) \in L^\varphi(\Omega)$,
for every $i=1,2\cdots,n$. Now for some $\beta>0$, we have
\begin{eqnarray*}
\int_\Omega \varphi\left(\beta
\bigg|\sum_{k=1}^n\left(\frac{\partial f}{\partial x_k}\circ
T\right)\frac{\partial T_k}{\partial x_i}\bigg|\right)d\mu & = &
\int_\Omega \varphi\left(\beta \bigg|\sum_{k=1}^n \frac{\partial
f}{\partial x_i}(x)\frac{\partial T_k}{\partial
x_i}(T^{-1}(x))\bigg|\right)\frac{d\mu\circ T^{-1}}{d\mu}d\mu\\
& = & \int_{\Omega\setminus \Omega_\circ} + \int_{\Omega_\circ}\\
&= & 0
\end{eqnarray*}
as $\frac{\partial f}{\partial x_k}(x)=0$ on $\Omega\setminus
\Omega_\circ$ for $k=1,2,\cdots,n$ and $\frac{d\mu\circ
T^{-1}}{d\mu}(x)=0$ on $\Omega_\circ$. This shows that
$\frac{\partial}{\partial x_i}(f\circ T) =0$ in $L^\varphi(\Omega)$
and hence $\big|\big|\frac{\partial}{\partial x_i}(f\circ
T)\big|\big|_\varphi =0$ for $i=1,2,\cdots,n$. Thus we have,
\begin{eqnarray*}
 ||f\circ
T||_{1,\varphi} & = & ||f\circ T||_\varphi +
\displaystyle{\sum_{i=1}
^n\bigg|\bigg|\frac{\partial}{\partial x_i}(f\circ T)\bigg|\bigg|_\varphi}\\
& = & 0
\end{eqnarray*}
Therefore, $f\circ T =0$ in $W^{1,\varphi}(\Omega)$ and
$W^{1,\varphi}(\Omega_\circ)\subseteq ker\, C_T$. Hence the result
follows.
\end{proof}
The following theorem gives a necessary and sufficient condition for
injectivity of composition operator $C_T$ on Orlicz-Sobolev space
$W^{1,\varphi}(\Omega)$. We say that $T: \Omega \rightarrow \Omega$
is essentially surjective if $\mu(\Omega\setminus T(\Omega))=0$.
\begin{thm}
The composition operator $C_T$ induced by $T$ on Orlicz-Sobolev
space $W^{1,\varphi}(\Omega)$ is injective if and only if $T$ is
essentially surjective.
\end{thm}
\begin{proof}
Suppose that $C_T$ is injective. Then $ker\, C_T =\{0\}$. But $ker\,
C_T= W^{1,\varphi}(\Omega_\circ)= \left\{ f\in
W^{1,\varphi}(\Omega): \frac{d\mu\circ
T^{-1}}{d\mu}\big|_{S_f}=0\right\}$. Therefore, $ker\, C_T=\{0\}$
implies that $W^{1,\varphi}(\Omega_\circ)=\{0\}$. This shows that
$f=0$ a.e. if $\frac{d\mu\circ T^{-1}}{d\mu}|_{S_f}=0$. Hence it
follows that $\frac{d\mu\circ T^{-1}}{d\mu}\neq 0$ a.e.. Thus
$\mu(\Omega_\circ)=0$. To complete the proof, it suffices to show
that $\Omega\setminus\Omega_\circ = T(\Omega)$. Note that
$\Omega\setminus\Omega_\circ= S_{\frac{d\mu\circ T^{-1}}{d\mu}}$.\\
Let $E\subset \Omega\setminus T(\Omega)$. Then $T^{-1}(E)=\emptyset$
and hence, $0= \mu(T^{-1}(E))= \displaystyle{\int_E \frac{d\mu\circ
T^{-1}}{d\mu}d\mu}$ implies that $\frac{d\mu\circ
T^{-1}}{d\mu}|_E=0$. This shows that $E\subset\Omega_\circ$ and
hence $\Omega\setminus T(\Omega)\subseteq\Omega_\circ$. Thus, we
have $$\mu(\Omega_\circ)=0 \Rightarrow \mu(\Omega\setminus
T(\Omega))=0.$$ Conversely, assume that $T$ is essentially
surjective so that $\Omega= T(\Omega)\cup A$, where $\mu(A)=0$.
Then, clearly, we have
\begin{eqnarray*}
ker \, C_T & = & \{ f\in W^{1,\varphi}(\Omega): C_T f=0\}\\
& = & \{ f\in W^{1,\varphi}(\Omega): f|_{T(\Omega)}=0\}\\
& = & \{0\}\,\,\,\,\,\,\,\,[\,\because \mu(A)=0\,]
\end{eqnarray*}
Therefore, $C_T$ is injective.
\end{proof}
\section{Ascent of the Composition Operator}
 \par F. Riesz in \cite{func13} introduced the concept of ascent and descent for a linear operator in a connection
with his investigation of compact linear operators. The study of
ascent and descent has been done as a part of spectral properties of
an operator (\cite{func14}, \cite{func15}). Before going to start,
let us recall the notion of ascent of an operator on an arbitrary
vector space $X$.\\
\par If $H: X \rightarrow X$ is an operator on $X$, then the null
space of $H^k$ is a $H$-invariant subspace of $X$, that is,
$H(ker\,(H^k))\subseteq ker\,(H^k)$ for every positive integer $k$.
Indeed, if $x\in \,ker\,(H^k)$ then $H^k(x)=0$ and therefore,
$H^k(H(x))=H(H^k(x))=0$, i.e., $H(x)\in ker\,(H^k)$. Thus we have
the following subspace inclusions:
$$ker\,(H)\subseteq ker\,(H^2) \subseteq ker(H^3)\subseteq \cdots$$
Following definitions and well known results are relevant to our
context (\cite{func11}, \cite{func}, \cite{func12});
\begin{thm}
 For an operator $H: X \rightarrow X$ on a vector space, if $ker\,(H^k)= ker\,(H^{k+1})$ for some $k$, then $ker\,(H^n)= ker(H^k)$ for all
$n\geq k$.
\end{thm}
We now introduce ascent of an operator. The ascent of $H$ is the
smallest natural number $k$ such that $ker\,(H^k)= ker\,(H^{k+1})$.
If there is no $k\in \mathbb{N}$ such that
 $ker\,(H^k)= ker\,(H^{k+1})$, then we say that ascent of $H$ is
 infinite.\\
\par Now we are ready to study ascent of the composition operator $C_T$ on Orlicz-Sobolev space $W^{1,\varphi}(\Omega)$.
Observe that if $T$ is a non singular measurable transformation on
$\Omega$, then $T^k$ is also non singular measurable transformation
for every $k\geq2$ with respect to the measure $\mu$. Hence $T^k$
also induces a composition operator $C_{T^k}$. Note that for every
measurable function $f$, $C^k_T(f)= f\circ T^k = C_{T^k}(f)$. Also
we have
$$\cdot \cdot\cdot \ll\mu\circ T^{-(k+1)}\ll\mu\circ T^{-k}\ll \cdot \cdot \cdot \ll\mu\circ T^{-1}\ll\mu .$$
Take $\mu \circ T^{-k}= \mu_k$. Then by Radon-Nikodym theorem, there
exists a non-negative locally integrable function $f_{T^k}$ on
$\Omega$ so that the measure $\mu_k$ can be represented as $$
\mu_k(A) = \int_A f_{T^{k}}(x) d\mu(x),
\,\,\,\mbox{for~~all~~$A\in\Sigma$}$$ where the function $f_{T^k}$
is the Radon-Nikodym derivative of the measure $\mu_k$ with respect
to the measure $\mu$. The following theorem characterizes the
composition operators $C_T$ with ascent $k$ on Orlicz-Sobolev spaces
$W^{1,\varphi}(\Omega)$.
\begin{thm}
The composition operator on Orlicz space $W^{1,\varphi}(\Omega)$ has
ascent $k\geq 1$ if and only if $k$ is the first positive integer
such that the measures $\mu_k$ and $\mu_{k+1}$ are equivalent.
\end{thm}
\begin{proof}
Suppose that $\mu_k$ and $\mu_{k+1}$ are equivalent. Then $\mu_{k+1}
\ll\mu_k\ll\mu_{k+1}$. Since $\mu_k\ll\mu_{k+1}\ll\mu$, hence the
chain rule of Radon-Nikodym derivative implies that
\begin{align}
\frac{d\mu_k}{d\mu}(x) & = \frac{d \mu_k}{d \mu_{k+1}}(x) \cdot
\frac{d\mu_{k+1}}{d\mu}(x)\\
\Rightarrow f_{T^k}(x) & =\frac{d \mu_k}{d \mu_{k+1}}(x)\cdot
f_{T^{k+1}}(x) \label{eq:1}
\end{align}
Similarly, $\mu_{k+1}\ll\mu_k\ll\mu$ implies that
\begin{eqnarray}
f_{T^{k+1}}(x)= \frac{d\mu_{k+1}}{d\mu_{k}}(x)\cdot
f_{T^{k}}(x)\label{eq:2}
\end{eqnarray}
Now, the kernel of $C^k_T$ given by $ker\,(C^k_T) = ker\,(C_{T^k})
=W^{1,\varphi}(\Omega_k)$, where $\Omega_k = \{x\in \Omega:
f_{T^k}(x)=0\}$. Similarly,  $ker\,(C^{k+1}_T) =
W^{1,\varphi}(\Omega_{k+1})$, where $\Omega_{k+1} = \{x\in \Omega:
f_{T^{k+1}}(x)=0\}$. From $\ref{eq:1}$ and $\ref{eq:2}$, it follows
that $\Omega_k= \Omega_{k+1}$. Therefore we have,
$$ ker\,(C^k_T)= W^{1,\varphi}(\Omega_k)= W^{1,\varphi}(\Omega_{k+1})=
ker\,(C^{k+1}_T).$$ Since $k$ is the least hence, the ascent of $C_T$ is $k$.\\
Conversely, suppose that ascent of $C_T$ is $k$. Now this implies
that if $ker\,(C^k_T)= W^{1,\varphi}(\Omega_k)$ and
$ker\,(C^{k+1}_T)=W^{1,\varphi}(\Omega_{k+1})$, then
$W^{1,\varphi}(\Omega_k)=W^{1,\varphi}(\Omega_{k+1})$. Hence
$\Omega_k= \Omega_{k+1}$ almost everywhere with respect to the
measure $\mu$. So $\Omega_k = \{x\in \Omega: f_{T^k}(x)=0\}=\{x\in
\Omega: f_{T^{k+1}}(x)=0\}$. It is known that $\mu_{k+1}\ll\mu_{k}$.
Thus only need to show $\mu_k \ll \mu_{k+1}$. For this let $E\in
\Sigma $ such
that $\mu_{k+1}(E)=0$. Now we have the following cases:\\
Case-$1$: When  $E\cap \Omega_k =\emptyset$.\\
Then $0= \mu_{k+1}(E)= \int_E f_{T^{k+1}}(x)d\mu(x)$ implies that
$\mu(E)=0$ as on $E$, $f_{T^{k+1}}(x)>0$. As $\mu_k(E)= \int_E
f_{T^k}(x) d\mu(x) $ and $\mu(E)=0$, hence $\mu_k(E)=0$.\\
Case-$2$: when $E\cap \Omega_k\neq \emptyset$.\\
Then we have,
\begin{align*}
0 =\mu_{k+1}(E) &  = \int_E
f_{T^{k+1}}(x)d\mu(x)\\
& = \int_{E\setminus (E\cap \Omega_k)}f_{T^{k+1}}(x)d\mu(x) +
\int_{E\cap \Omega_k}f_{T^{k+1}}(x)d\mu(x)\\
& = \int_{E\setminus (E\cap \Omega_k)}f_{T^{k+1}}(x)d\mu(x)
\end{align*}
Now this implies that $\mu(E\setminus (E\cap \Omega_k))=0$.
Therefore, in either cases $\mu_{k+1}(E)=0$ implies that
$\mu_k(E)=0$. Thus $\mu_{k+1}\ll\mu_k\ll\mu_{k+1}$.
\end{proof}
\begin{cor}
Ascent of the composition operator $C_\tau$ on Orlicz-Sobolev spaces
is infinite if and only if there does not exist any positive integer
$k$ such that the measures $\mu_k$ and $\mu_{k+1}$ are equivalent.
\end{cor}
\par We say that a measurable transformation $T$ is measure preserving if $\mu(T^{-1}(E)) =
\mu(E)$ for all $E\in \Sigma$. We also have the following results:
\begin{cor}
\begin{enumerate}
\item  If the measure $\mu$ is measure preserving then the ascent of
the composition operator $C_T$ on Orlicz-Sobolev space
$W^{1,\varphi}(\Omega)$ is $1$.
\item If $T$ is a nonsingular surjective measurable transformation such that
$\mu(\tau^{-1}(E))\geq \mu(E)$ for all $E\in \Sigma$, then also the
ascent of the composition operator induced by $T$ on Orlicz-Sobolev
space is $1$.
\item If $T$ is essentially surjective, then also ascent of $C_T$ is
equal to $1$.
\end{enumerate}
\end{cor}
\section*{\small Conclusions:} We have proposed and proved a necessary
and sufficient condition for the injectivity of composition operator
$C_T$. We have also characterized the operator $C_T$ defined on
Orlicz-Sobolev space with finite and infinite ascent. Our future
plan of work will be to apply these results to a class of non-linear
PDEs.
\section*{\small Acknowledgement:}
One of the author R.K. Giri acknowledge the Ministry of Human
Resource Development (M.H.R.D.), India for the financial assistantship.

\end{document}